\documentclass[12pt]{amsart}

\title[Every model embeds into its constructible universe]{Every countable model of set theory embeds into its own constructible universe}

\author{Joel David Hamkins}

\address{Mathematics, Philosophy, Computer Science,
         The Graduate Center of The City University of New York,
         365 Fifth Avenue, New York, NY 10016, U.S.A. \&
         Department of Mathematics,
         The College of Staten Island of CUNY,
         Staten Island, NY 10314, U.S.A.}

\email{jhamkins@gc.cuny.edu}
\urladdr{http://jdh.hamkins.org}

\thanks{My research has been supported in part by NSF grant DMS-0800762, PSC-CUNY grant 64732-00-42 and Simons
Foundation grant 209252. In addition, I gratefully acknowledge the generous support provided to me as a visiting fellow in Spring 2012 at the Isaac Newton Institute for Mathematical Sciences in Cambridge, U.K., where most of this research was undertaken. I thank Philip Welch for helpful discussions undertaken whilst strolling the Churchill College grounds in Cambridge and a math-filled excursion to Wells in southern England, with cream tea, a Saxon church and mathematics in Bradford-on-Avon. Thanks also to Ali Enayat, Victoria Gitman and Emil \Jerabek\ for helpful comments.}

\usepackage{amssymb}
\usepackage[figuresright]{rotating}
\usepackage{tikz} 
\usetikzlibrary{arrows}
\usepackage{fullpage}
\newtheorem{theorem}{Theorem}
\newtheorem{maintheorem}[theorem]{Main Theorem}
\newtheorem{corollary}[theorem]{Corollary}
\newtheorem{sublemma}{Lemma}[theorem]
\newtheorem{lemma}[theorem]{Lemma}

\newtheorem{question}[theorem]{Question}

\newtheorem{observation}[theorem]{Observation}
\newtheorem{proposition}[theorem]{Proposition}

\newtheorem{definition}[theorem]{Definition}

\newcommand{\QED}{\end{proof}}

\def\proclaim[#1]{{\bf #1}}
\def\BF#1.{{\bf #1.}}

\newcommand{\Fraisse}{Fra\"\i ss\'e}
\newcommand{\Godel}{G\"odel}

\newcommand{\Jerabek}{Je\v r\'abek}

\newcommand{\N}{{\mathbb N}}
\renewcommand{\P}{{\mathbb P}}
\newcommand{\Q}{{\mathbb Q}}

\newcommand{\of}{\subseteq}

\newcommand{\set}[1]{\{\,{#1}\,\}}
\newcommand{\singleton}[1]{\left\{{#1}\right\}}

\newcommand{\rank}{\mathop{\rm rank}}

\newcommand{\Add}{\mathop{\rm Add}}

\newcommand{\restrict}{\upharpoonright} 
\newcommand{\satisfies}{\models}

\newcommand{\union}{\cup}

\newcommand{\Union}{\bigcup}

\newcommand{\intersect}{\cap}

\newcommand{\trianglelt}{\lhd}

\newcommand{\smalllt}{\mathrel{\mathchoice{\raise2pt\hbox{$\scriptstyle<$}}{\raise1pt\hbox{$\scriptstyle<$}}{\raise0pt\hbox{$\scriptscriptstyle<$}}{\scriptscriptstyle<}}}
\newcommand{\smallleq}{\mathrel{\mathchoice{\raise2pt\hbox{$\scriptstyle\leq$}}{\raise1pt\hbox{$\scriptstyle\leq$}}{\raise1pt\hbox{$\scriptscriptstyle\leq$}}{\scriptscriptstyle\leq}}}

\newcommand{\ltkappa}{{{\smalllt}\kappa}}

\newcommand{\boolval}[1]{\mathopen{\lbrack\!\lbrack}\,#1\,\mathclose{\rbrack\!\rbrack}}
\def\[#1]{\boolval{#1}}
\newbox\gnBoxA
\newdimen\gnCornerHgt
\setbox\gnBoxA=\hbox{\tiny$\ulcorner$}
\global\gnCornerHgt=\ht\gnBoxA
\newdimen\gnArgHgt
\def\gcode #1{%
\setbox\gnBoxA=\hbox{$#1$}%
\gnArgHgt=\ht\gnBoxA%
\ifnum     \gnArgHgt<\gnCornerHgt \gnArgHgt=0pt%
\else \advance \gnArgHgt by -\gnCornerHgt%
\fi \raise\gnArgHgt\hbox{\tiny$\ulcorner$} \box\gnBoxA %
\raise\gnArgHgt\hbox{\tiny$\urcorner$}}
\newcommand{\UnderTilde}[1]{{\setbox1=\hbox{$#1$}\baselineskip=0pt\vtop{\hbox{$#1$}\hbox to\wd1{\hfil$\sim$\hfil}}}{}}
\newcommand{\Undertilde}[1]{{\setbox1=\hbox{$#1$}\baselineskip=0pt\vtop{\hbox{$#1$}\hbox to\wd1{\hfil$\scriptstyle\sim$\hfil}}}{}}
\newcommand{\undertilde}[1]{{\setbox1=\hbox{$#1$}\baselineskip=0pt\vtop{\hbox{$#1$}\hbox to\wd1{\hfil$\scriptscriptstyle\sim$\hfil}}}{}}
\newcommand{\UnderdTilde}[1]{{\setbox1=\hbox{$#1$}\baselineskip=0pt\vtop{\hbox{$#1$}\hbox to\wd1{\hfil$\approx$\hfil}}}{}}
\newcommand{\Underdtilde}[1]{{\setbox1=\hbox{$#1$}\baselineskip=0pt\vtop{\hbox{$#1$}\hbox to\wd1{\hfil\scriptsize$\approx$\hfil}}}{}}

\newcommand{\st}{\mid}
\renewcommand{\th}{{\hbox{\scriptsize th}}}

\newcommand{\Implies}{\mathrel{\Rightarrow}}
\renewcommand{\iff}{\mathrel{\leftrightarrow}}
\newcommand{\Iff}{\mathrel{\Leftrightarrow}}

\newcommand{\iso}{\cong}
\def\<#1>{\langle#1\rangle}

\newcommand{\Ord}{\mathop{{\rm Ord}}}
\newcommand{\No}{\mathop{{\rm No}}} 

\newcommand{\HC}{\mathop{{\rm HC}}}

\newcommand{\ZFC}{{\rm ZFC}}
\newcommand{\ZF}{{\rm ZF}}

\newcommand{\GB}{{\rm GB}}

\newcommand{\KP}{{\rm KP}}

\newcommand{\HOD}{{\rm HOD}}

\newcommand{\HF}{{\rm HF}}

\newcommand{\PA}{{\rm PA}}

\newcommand{\cell}[1]{\boxit{\hbox to 17pt{\strut\hfil$#1$\hfil}}}
\newcommand{\head}[2]{\lower2pt\vbox{\hbox{\strut\footnotesize\it\hskip3pt#2}\boxit{\cell#1}}}
\newcommand{\boxit}[1]{\setbox4=\hbox{\kern2pt#1\kern2pt}\hbox{\vrule\vbox{\hrule\kern2pt\box4\kern2pt\hrule}\vrule}}
\newcommand{\Col}[3]{\hbox{\vbox{\baselineskip=0pt\parskip=0pt\cell#1\cell#2\cell#3}}}
\newcommand{\tapenames}{\raise 5pt\vbox to .7in{\hbox to .8in{\it\hfill input: \strut}\vfill\hbox to
.8in{\it\hfill scratch: \strut}\vfill\hbox to .8in{\it\hfill output: \strut}}}
\newcommand{\Head}[4]{\lower2pt\vbox{\hbox to25pt{\strut\footnotesize\it\hfill#4\hfill}\boxit{\Col#1#2#3}}}
\newcommand{\Dots}{\raise 5pt\vbox to .7in{\hbox{\ $\cdots$\strut}\vfill\hbox{\ $\cdots$\strut}\vfill\hbox{\
$\cdots$\strut}}}
\newcommand{\df}{\it} 
\newcommand{\toward}{\rightharpoonup}
\newcommand{\ttoward}{\twoheadrightarrow}
\newcommand{\proj}{\mathrel{\;\lower2pt\hbox{\turnbox{45}{$\leftharpoonup$}}\;\;}}

\newcommand{\Hg}{\mathop{{\rm Hg}}}

\begin{document}

\maketitle

\begin{abstract}
The main theorem of this article is that every countable model of set theory $\<M,{\in^M}>$, including every well-founded model, is isomorphic to a submodel of its own constructible universe $\<L^M,{\in^M}>$ by means of an embedding $j:M\to L^M$. It follows from the proof that the countable models of set theory are linearly pre-ordered by embeddability: if $\<M,{\in^M}>$ and $\<N,{\in^N}>$ are countable models of set theory, then either $M$ is isomorphic to a submodel of $N$ or conversely. Indeed, these models are pre-well-ordered by embeddability in order-type exactly $\omega_1+1$. Specifically, the countable well-founded models are ordered under embeddability exactly in accordance with the heights of their ordinals; every shorter model embeds into every taller model; every model of set theory $M$ is universal for all countable well-founded binary relations of rank at most $\Ord^M$; and every ill-founded model of set theory is universal for all countable acyclic binary relations. Finally, strengthening a classical theorem of Ressayre, the proof method shows that if $M$ is any nonstandard model of \PA, then every countable model of set theory---in particular, every model of \ZFC\ plus large cardinals---is isomorphic to a submodel of the hereditarily finite sets $\<\HF^M,{\in^M}>$ of $M$. Indeed, $\<\HF^M,{\in^M}>$ is universal for all countable acyclic
binary relations.
\end{abstract}


\section{Introduction}

In this article, I shall prove that every countable model of set theory $\<M,{\in^M}>$, including every well-founded model, is isomorphic to a submodel of its own constructible universe $\<L^M,{\in^M}>$. In other words, there is an embedding
$j:\<M,{\in^M}>\to \<L^M,{\in^M}>$.
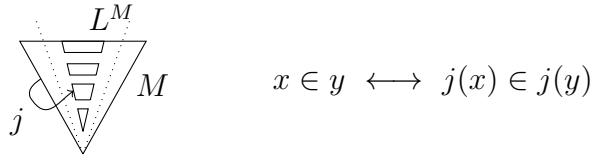
$$\hbox{\begin{tikzpicture}[xscale=.07,yscale=.3]
 \draw (-0,0) --(12,5) --(-12,5) --(0,0);
 \draw[dotted] (0,0) --(9,6);
 \draw[dotted] (0,0) --(-9,6);
 \node[anchor=south west] at (-1,5) {$L^M$};
 \draw (0,1) --(1,2) --(-1,2) --(0,1);
 \draw (1.4,2.4) --(2.1,3.1) --(-2.1,3.1) --(-1.4,2.4) --(1.4,2.4);
 \draw (2.5,3.5) --(3,4) --(-3,4) --(-2.5,3.5) --(2.5,3.5);
 \draw (3.5,4.5) --(4,5) --(-4,5) --(-3.5,4.5) --(3.5,4.5);
 \draw[->] (-8,3.33) to [out=190,in=150] (-10,2.5) to [out=-20,in=190] (-1.8,2.8);
 \node[anchor=north east] at (-9,2.5) {$j$};
 \node[anchor=north west] at (8,4) {$M$};
\end{tikzpicture}}\qquad\quad\raise 25pt\hbox{$x\in y\ \longleftrightarrow\ j(x)\in j(y)$}$$

\begin{maintheorem}\label{MainTheorem.MsubmodelL^M}
 Every countable model of set theory $\<M,{\in^M}>$ is isomorphic to a submodel of its own constructible universe $\<L^M,{\in^M}>$.
\end{maintheorem}

The proof uses universal digraph combinatorics, including an acyclic version of the countable random digraph, which I call the countable random $\Q$-graded digraph, and higher analogues arising as uncountable \Fraisse\ limits, leading eventually to what I call the hypnagogic digraph, a set-homogeneous, class-universal, surreal-numbers-graded acyclic class digraph, which is closely connected with the surreal numbers. The proof shows that $\<L^M,{\in^M}>$ contains a submodel that is a universal acyclic digraph of rank $\Ord^M$, and so in fact this model is universal for all countable acyclic binary relations of this rank. When $M$ is ill-founded, this includes all acyclic binary relations. The method of proof also establishes the following theorem, thereby answering a question posed by Ewan Delanoy \cite{MSE64095Delanoy:ComparingCountableModelsOfZFC}.

\begin{maintheorem}\label{MainTheorem.CountableModelsLinearlyOrdered}
 The countable models of set theory are linearly pre-ordered by embeddability: if $\<M,{\in^M}>$ and $\<N,{\in^N}>$ are countable models of set theory, then either $M$ is isomorphic to a submodel of $N$ or conversely. Indeed, the countable models of set theory are pre-well-ordered by embeddability in order type exactly $\omega_1+1$.
\end{maintheorem}

The proof shows that the embeddability relation on the countable models of set theory conforms with their ordinal heights, in that any two countable models with the same ordinals are bi-embeddable; any shorter model embeds into any taller model; and the ill-founded models are all bi-embeddable and universal.

The proof method arises most easily in {\it finite} set theory, showing that the nonstandard hereditarily finite sets $\HF^M$ coded in any nonstandard model $M$ of \PA\ or even of $I\Delta_0$ are similarly universal for all acyclic binary relations. This strengthens a classical theorem of Ressayre, while simplifying the proof, replacing a partial saturation and resplendency argument with a soft appeal to graph universality.

\begin{maintheorem}\label{MainTheorem.NonstandardModelsOfFiniteSetTheoryAreUniversal}
If $M$ is any nonstandard model of \PA, then every countable model of set theory is isomorphic to a
submodel of the hereditarily finite sets $\<\HF^M,{\in^M}>$ of $M$. Indeed, $\<\HF^M,{\in^M}>$ is universal for all countable acyclic
binary relations.
\end{maintheorem}

In particular, every countable model of \ZFC\ arises as a submodel of $\<\HF^M,{\in^M}>$. Thus, inside any nonstandard model of finite set theory, we may cast out some of the finite sets and thereby arrive at a copy of any desired model of infinite set theory, having infinite sets, uncountable sets or even large cardinals of whatever type we like.

A structure $M$ is {\df universal} for a class $\Delta$ of structures, if every structure in $\Delta$ is isomorphic to a substructure of $M$. In this article I use the term {\df model of set theory} to mean a first-order structure $\<M,{\in^M}>$ satisfying at least the Kripke-Platek \KP\ axioms of set theory, a very weak fragment of \ZF. Although typically for the models of set theory I have in mind the models of \ZFC\ or even \ZFC\ plus large cardinals, it turns out that the results of this article go through for the much weaker theory \KP, and indeed still weaker theories will suffice, but for definiteness I shall use \KP, leaving the determination of how weak we may go for a later project.

The language of set theory has only the set-membership relation $\in$, and so a {\df submodel} of a model $\<M,{\in^M}>$ of set theory is simply a subset $N\of M$, where one restricts the relation to form the structure \hbox{$\<N,{\in^M}\restrict N>$}. An {\df embedding} of one model $\<M,{\in^M}>$ into another $\<N,{\in^N}>$ is an isomorphism of $M$ with a submodel of $N$, that is, a function $j:M\to N$ for which $x\in^M y$ if and only if $j(x)\in^N j(y)$. Since a model of set theory is a set with a binary relation, it is technically a certain special kind of directed graph, and many of the arguments of this article will proceed from this graph-theoretic perspective. For example, a submodel of a model of set theory, viewed as a directed graph, is just an induced subgraph.

The three main theorems reappear in this article as theorems \ref{Theorem.MisSubmodelOfL^M}, \ref{Theorem.CountableModelsLinearlyOrdered} and \ref{Theorem.NonstandardModelsAreUniversal}, respectively.

\section{The countable random $\Q$-graded digraph}

The main theorems will be proved by finding copies of certain universal digraphs among the submodels of the models of set theory under consideration. So let me begin by developing a little of this universal digraph theory. A {\df digraph}, or directed graph, is a structure $\<G,{\toward}>$ where $G$ is a set of vertices, or nodes, and $\toward$ is a binary relation on $G$. A digraph is {\df acyclic} if there is no finite directed path from a vertex to itself. That is, an acyclic digraph is one with no directed cycles. Note that the undirected version of an acyclic digraph may be far from a tree, since there can be undirected cycles.

Let $\Q$ be the endless dense linear order of the rational numbers. A digraph $G$ is {\df $\Q$-graded} if every vertex $a$ in $G$ is assigned a rational value $q_a$ in such a way that $q_a<q_b$ whenever $a\toward b$. Such a graph must be acyclic, since the values increase along any directed path. The grading of a digraph is in effect a laying-out of its vertices on levels, in such a way that every directed edge points from a lower-level node to a higher-level node. Similarly, for any linear order $\ell$ we may consider the $\ell$-graded digraphs, with values in $\ell$, and these also are acyclic.

We may similarly grade an acyclic digraph using only relative grading values, defining that a {\df graded digraph} is a first-order structure of the form $\<G,{\toward},{\leq}>$, where $\<G,\toward>$ is a digraph and $\leq$ is a linear pre-order on the vertices whose corresponding strict order $<$ includes the edge relation, so that $a\toward b$ implies $a<b$. Every such structure can naturally be regarded as an $\ell$-graded digraph, where $\ell$ is the induced linear order on the equivalence classes of $\leq$, determined by the relation $a\equiv b\Iff a\leq b\leq a$. The graded digraph $\<G,\toward,\leq>$ is said to be {\df strictly} graded when $\leq$ is a linear order, rather than merely a pre-order, and similarly, an $\ell$-graded digraph is strictly $\ell$-graded when it has at most one node of any given value.

\begin{lemma}\label{Lemma.CycleFreeImpliesQgraded}
 Every acyclic digraph can be graded, and indeed, strictly graded. Every countable acyclic digraph can be strictly $\Q$-graded.
\end{lemma}

\begin{proof}
Suppose that $G$ is an acyclic digraph. Let $\trianglelt$ be the reachability relation on $G$, the transitive closure of the edge relation of $G$. Because $G$ is acyclic, it follows that $\trianglelt$ is a partial order on the vertices. Since every partial order extends to a
linear order, there is a linear order relation $\leq$ on $G$ such that $a\toward b\Implies a\leq b$. So we have found the desired strict grading $\<G,\toward,\leq>$. Since $\Q$ is universal for countable linear orders, if $G$ is countable then we may find a copy of the linear order $\leq$ in $\Q$ and thereby injectively assign the vertices of $G$ to rational numbers in such a way that $a\trianglelt b$ implies $q_a<q_b$. In particular, this is a strict $\Q$-grading of $G$.
\end{proof}

My interest in graded digraphs arises from the observation that every model of set theory $\<M,{\in}^M>$ is an acyclic digraph with a natural grading defined by $a\leq b\Iff\rank(a)\leq\rank(b)$, where $\rank(a)$ is the von Neumann rank, the smallest ordinal $\alpha$ for which $a\in V_{\alpha+1}^M$. This is a grading because $a\in b$ implies $\rank(a)<\rank(b)$. Thus, every model of set theory $M$ is naturally an $\Ord^M$-graded digraph.

\begin{observation}\label{Observation.ModelOfSetTheoryAsOrdGradedDigraph}
Every model of set theory  $\<M,{\in^M}>$ is an $\Ord^M$-graded digraph, using von Neumann rank as values.
\end{observation}

Since these graded digraphs are never strict---and they are the relevant digraphs to consider for the main theorem---we shall henceforth in this article be more interested in non-strict gradings. In particular, the universal graded digraphs constructed later in this article will have infinitely many nodes of each value.

A structure $A$ is {\df homogeneous} if every isomorphism of finitely generated substructures of $A$ extends to an automorphism of $A$.

\begin{theorem}\label{Theorem.TheCountableRandomQgradedDigraph}
There is a countable homogeneous $\Q$-graded digraph $\Gamma$, which is universal for all countable $\Q$-graded digraphs. Furthermore,
 \begin{enumerate}
  \item There is a unique such digraph up to isomorphism, even when restricting universality to finite $\Q$-graded digraphs.
  \item Every countable acyclic digraph arises as an isomorphic copy of an induced subgraph of $\Gamma$.
  \item $\Gamma$ admits a computable presentation.
 \end{enumerate}
\end{theorem}

Let us call this $\Gamma$ the \emph{countable random $\Q$-graded digraph}.

\begin{proof}
There are several independent constructions of this highly canonical object. Let me begin with the abstract realization of $\Gamma$ as a
\Fraisse\ limit. Let $\mathcal{Q}$ be the collection of finite graded digraphs. Note that $\mathcal{Q}$ enjoys the {\df hereditary} property, that any induced subgraph of a graph in $\mathcal{Q}$ is in $\mathcal{Q}$, and also the {\df joint embedding} property, asserting that any two members of $\mathcal{Q}$ embed into a third. Furthermore, $\mathcal{Q}$ has the {\df amalgamation} property, meaning that whenever $A$, $B$ and $C$ are in $\mathcal{Q}$, where $A$ embeds into $B$ and also into $C$, then there is a $D$ into which both $B$ and $C$ embed while agreeing on the image of $A$. These three properties are sufficient that $\mathcal{Q}$ has a \Fraisse\ limit $\Gamma$, a countable, homogeneous, directed graded graph, whose finite induced graded subdigraphs are up to isomorphism precisely the graded digraphs in $\mathcal{Q}$ (see \cite[thm\,7.1.2]{Hodges1993:ModelTheory}). The universal property of the \Fraisse\ limit ensures that $\Gamma$ is universal for all countable graded digraphs, since any such graph $G$ is the direct limit of its finite induced subgraphs, which are all in $\mathcal{Q}$, and one may therefore amalgamate the limit construction of $G$ as occurring inside $\Gamma$. Note that the pre-order on $\Gamma$ will be a countable dense linear pre-order, and so we may regard $\Gamma$ as $\Q$-graded. Furthermore, as a $\Q$-graded digraph, a simple homogeneity argument shows that $\Gamma$ remains universal for all countable $\Q$-graded digraphs. That is, we may in the universal property assume that the embeddings respect not only the relative grading, but also respect the absolute grading value in $\Q$ that we have just fixed. In particular, any countable homogeneous graded digraph universal for all finite graded digraphs can be assigned a $\Q$-grading for which it is also universal for all countable $\Q$-graded digraphs in this stronger sense, and furthermore it is homogeneous not only as a graded digraph but as a $\Q$-graded digraph, where the automorphisms respect the grading values.

Consider alternatively the following construction of $\Gamma$ via forcing, which amounts essentially to the forcing construction of the \Fraisse\ limit. Let $\P$ be the partial order of all finite $\Q$-graded digraphs with nodes coming from a fixed countably infinite set, ordering the digraphs of $\P$ by the induced subgraph relation. If $A\of B$ and $C$ are any such finite $\Q$-graded digraphs, with an embedding $f:A\to C$, then let $D_{A,B,C,f}$ be the set of graphs $H\in\P$ such that if $C\of H$, then $f$ extends to an embedding $g:B\to H$. This is a dense collection of conditions $H$, since for any $H_0\in \P$, if $C\of H_0$ then we can extend $H_0$ to an $H$ that includes a copy of $B$ over $f(A)$, and thus extend $f$ to $g:B\to H$, as desired. Since there are only countably many such dense sets, we may build a chain of graphs $H_0\of H_1\of \cdots$ such that $\Gamma=\Union_n H_n$ meets every such dense set. It follows that $\Gamma$ is {\df weakly homogeneous}, which means that whenever $A\of B$ and $f:A\to \Gamma$, then there is an extension to $g:B\to \Gamma$. This implies that $\Gamma$ is universal for all countable $\Q$-graded digraphs, since any such graph $G$ is the union $\Union_n G_n$ of finite graphs, and weak homogeneity allows us to build up a chain of embeddings $f_n:G_n\to \Gamma$, which realizes $G$ as an induced subgraph of $\Gamma$. It also implies that $\Gamma$ is fully homogeneous, by a similar back-and-forth argument that systematically extends a finite partial isomorphism to a full automorphism of $\Gamma$.

Let me now describe a more concrete construction, which will realize $\Gamma$ as a computable graded digraph. I shall build $\Gamma$ as the union of a sequence of finite graded digraphs $\Gamma_n\in Q$. Namely, let $\Gamma_0$ be the empty graph, and at each
stage $n$, consider all possible ways (up to isomorphism) to extend $\Gamma_n$ by the addition of one vertex---there are only finitely many possibilities---and let $\Gamma_{n+1}$ have points realizing all such types of connectivity patterns and grading relations with respect to the points in $\Gamma_n$. Since we can carry out this process by a computational procedure, the limit $\Gamma=\Union_{n\in\omega}\Gamma_n$ is a computable graded digraph. Furthermore, $\Gamma$ is clearly universal by the ``forth'' construction: given any countable graded digraph $G$, we enumerate the vertices of $G$ as $v_0$, $v_1$, and so on, and at stage $n$ we have a copy of the vertices of $G$ before $a_n$ inside $\Gamma_n$. The next vertex $v_n$ has a certain type of relation to the previous vertices with respect to the edge and relative value relations, and we added a point to $\Gamma_{n+1}$ exhibiting precisely that pattern of connectivity and grading relation to the corresponding images of those points. Thus, we may extend our embedding one more step and ultimately construct an embedding of $G$ into $\Gamma$, as desired.

%
%
%
Finally, let me give a probabilistic account of $\Gamma$. Let $\Gamma$ have infinitely many vertices of every rational value, and for each pair of nodes $a$ and $b$ with $a$ having lower value than $b$, flip a coin to determine whether we place an edge between them. Thus, $\Gamma$ is the random $\Q$-graded digraph. With probability one, every finite $\Q$-graded digraph arises as an induced sub graph, and
furthermore, any particular finite subgraph $A$ will find with probability one an extension inside $\Gamma$ realizing any particular
pattern of connectivity for a new vertex having a particular value (subject to the constraints on values). Thus, almost surely the
resulting graph $\Gamma$ will exhibit the finite-pattern property that allows us to extend any partial isomorphism via a back-and-forth argument to a full automorphism.
\end{proof}

The need to consider graded digraphs, rather than mere digraphs, arises from the fact that there simply are no nontrivial homogeneous acyclic digraphs:

\begin{observation}\label{Observation.NoHomogeneousDigraphWithoutGrading}
 If an acyclic digraph $G$ has $a\toward b\toward c$ as an induced subgraph, then it is not homogeneous as a digraph (without any grading structure).
\end{observation}

\begin{proof}
Suppose $G$ is an acyclic digraph with $a\toward b\toward c$ as an induced subgraph, meaning here that $a$ and $c$ share no edge. It
follows that the subgraph $\{a,c\}$ is discrete and therefore has an automorphism swapping $a$ and $c$. But this finite automorphism cannot
extend to an automorphism $\pi$ of the whole graph, since this would cause a directed cycle: $a\toward b\toward c=\pi(a)\toward
\pi(b)\toward \pi(c)= a$, a contradiction.
\end{proof}

Essentially the same point is made by observing that the collection of acyclic digraphs does not enjoy the amalgamation property, since the discrete graph with two vertices embeds into $a\toward b\toward c$ in two ways, but these cannot be amalgamated in any acyclic graph, since any amalgamation would give rise to a directed cycle.

The purpose of using graded digraphs is to overcome the obstacle posed by observation \ref{Observation.NoHomogeneousDigraphWithoutGrading} and enable the theory of universality and homogeneity in the context of acyclic digraphs. The situation here is that the countable random $\Q$-graded digraph is homogeneous as an $\Q$-graded digraph, where the automorphisms respect the edge relation and the values of the $\Q$-grading, but not as an unlabeled digraph, where the automorphism need only respect the edge relation. Note that for $a\toward b\toward c$ in the case of an $\Q$-graded graph, the value of $a$ must be strictly smaller than the value of $c$, which exactly prevents us in swapping $a$ and $c$ when we must preserve the values. Meanwhile, if one forgets both the values and the direction of edges, then in fact the undirected graph underlying the countable random $\Q$-graded digraph is the same as the countable random graph.

\begin{observation}\label{Observation.CountableRandomDigraphIsCountableRandomGraph}
 The graph obtained by ignoring the values and direction of the edges in the countable random $\Q$-graded digraph is isomorphic to the countable random graph.
\end{observation}

\begin{proof}
It suffices to show that the graph underlying the countable random $\Q$-graded digraph $\Gamma$ exhibits the finite-pattern property characterizing the countable random graph, namely, that for any two disjoint sets $A$ and $C$, there is a node connected to every edge in $A$ and to no edge in $C$. Such a node exists in $\Gamma$ by the universality and homogeneity of $\Gamma$: simply select a rational value not used by any node in $A$ and observe that there is a larger graded graph realizing the desired pattern with respect to $A$ and $C$ and having that value; by universality there is a copy of this digraph inside $\Gamma$; by homogeneity, we may translate this copy to align with the given nodes of $A$ and $C$. So the undirected ungraded digraph underlying $\Gamma$ has the finite-pattern property, and consequently is isomorphic to the countable random graph, since the countable random graph has this property and a simple back-and-forth argument shows that any two countable graphs with this property are isomorphic.
\end{proof}

\section{A strengthening of Ressayre's theorem}

Every model of arithmetic comes along with a model of finite set theory via the Ackerman coding, obtained by defining the
following relation on the natural numbers:
$$n\mathrel{E}m\quad\Iff\quad\text{ the }n^\th\text{ binary bit of }m\text{ is }1.$$
In the standard model $\N$, this relation is well-founded and extensional, and it is an often-assigned elementary exercise to prove that
the Mostowski collapse of $E$ is precisely the hereditary finite sets.
 $$\<\N,E>\quad \iso\quad \<\HF,\in>$$
The relation $E$ can be defined inside any model of arithmetic $M$, leading to the corresponding model of finite set theory $\<\HF^M,{\in^M}>$, and conversely, the natural numbers of $\HF^M$ are isomorphic to $M$ again. In this way, the models of arithmetic are in a natural one-to-one correspondence with the models of finite set theory (properly understood), and the two theories are bi-interpretable. Specifically, \cite{KayeWong2007:OnInterpreationsOfArithmeticAndSetTheory} shows that the models of \PA\ give rise to all the models of the finite set theory $\ZFC^{\neg\infty}=\ZFC-\text{Inf}+\neg\text{Inf}+\text{TC}$, where $\text{TC}$ is the assertion that every set has a transitive closure, and conversely the natural numbers of any such model satisfies \PA, in inverse fashion, making for a bi-interpretation of these two theories. The curious anomaly here is that while the axiom $\text{TC}$ is provable and may be omitted if one replaces the usual foundation axiom in \ZFC\ with the axiom scheme of $\in$-induction, nevertheless results in \cite{EnayatSchmerlVisser2011:OmegaModelsOfFiniteSetTheory} show, quite interestingly, that $\text{TC}$ is not provable if one uses only the usual foundation axiom.

The theme of this article begins in earnest with the following remarkable theorem of Ressayre's.

\begin{theorem}[Ressayre \cite{Ressayre1983:IntroductionToRecursivelySaturatedModels}]\label{Theorem.Ressayre'sTheorem}
If $M$ is any nonstandard model of \PA, with $\<\HF^M,{\in^M}>$ the corresponding nonstandard hereditary finite sets of $M$, then for any
consistent computably axiomatized theory $T$ extending $\ZF$ in the language of set theory, there is a submodel $N\of\<\HF^M,{\in^M}>$ such that $N\satisfies T$.
\end{theorem}

In particular, we may find models of \ZFC\ and even of \ZFC\ + large cardinals as submodels of $\HF^M$, a structure whose theory is all about the finite and which thinks every object is finite.
Incredible! How can this be? Ressayre's proof uses partial saturation and resplendency to prove that one can find the submodel of the desired theory $T$.

Theorem \ref{Theorem.NonstandardModelsAreUniversal} strengthens Ressayre's theorem, while simplifying the proof, by replacing the use of resplendency with a soft appeal to digraph universality. In particular, the role of the theory $T$ is omitted: we need not assume that $T$ is computable, and we don't just get one model of $T$, but rather all countable models of $T$, for the theorem shows that the nonstandard models of finite set theory are universal for all countable acyclic binary relations. In particular, every model of set theory arises as a submodel of $\<\HF^M,{\in^M}>$.

\goodbreak
\begin{theorem}[Main theorem \ref{MainTheorem.NonstandardModelsOfFiniteSetTheoryAreUniversal}]\label{Theorem.NonstandardModelsAreUniversal}
If $M$ is any nonstandard model of \PA\ or indeed merely of $IE_1$, then every countable model of set theory is isomorphic to a
submodel of the hereditarily finite sets $\<\HF^M,{\in^M}>$ of $M$. Indeed, $\<\HF^M,{\in^M}>$ is universal for all countable acyclic
binary relations.
\end{theorem}

One may equivalently cast the theorem entirely in terms of finite set theory: every nonstandard model $\<M,{\in}^M>$ of finite
set theory $\ZFC^{\neg\infty}$ is universal for all countable acyclic binary relations. In particular, every model of set theory, of \ZFC\
or whatever theory, is isomorphic to a submodel of $\<M,{\in}^M>$, even though this latter model believes that every set is finite.

In order to prove theorem \ref{Theorem.NonstandardModelsAreUniversal}, I shall begin with a simple link between acyclic digraphs and sets with the set-membership relation.

\begin{lemma}\label{Lemma.FiniteDigraphsAsSets}
Every finite acyclic digraph is isomorphic to a hereditarily finite set with the set-membership relation $\in$.
\end{lemma}

\begin{proof}
Suppose that $G$ is a finite acyclic digraph with $n$ vertices $v_1$,\ldots$v_n$. I shall first extend $G$ to a larger finite extensional
digraph $G^+$ as follows. Let $N$ be the transitive closure of the digraph $$0\toward 1\toward \cdots\toward n,$$ which is the same as
$(n+1,\in)$, if we use the finite ordinals as natural numbers, and let $G^+$ be the graph containing the disjoint union of $G$
and $N$, together with edges $k\toward v_k$ for $k\geq 1$. Note that (i) any two nodes of $G$ have different predecessors in $N$; (ii)
different nodes of $N$ have different predecessors in $N$; and (iii) every nonzero node in $N$ has $0$ as a predecessor, but no node in $G$ has $0$
as a predecessor. Thus, no two nodes of $G^+$ have the same predecessors and so $G^+$ is extensional. Furthermore, since $G$ and $N$ are
both acyclic and I added edges pointing only from $N$ to $G$, it follows that $G^+$ is acyclic. In other words, the edge relation of $G^+$
is an extensional well-founded relation and therefore by the Mostowski collapse $\pi(v)=\set{\pi(w)\st w\toward v}$ is isomorphic to a
unique transitive set under the set membership relation $\in$.
 $$w\toward v\qquad\text{ if and only if }\qquad \pi(w)\in\pi(v).$$
This isomorphism associates each node in $G^+$ with the set of sets associated with its children. The isomorphism $\pi$ carries the
directed graph $(G,\toward)$ to the hereditarily finite set $(A,\in)$, where $A=\set{\pi(v)\st v\in G}$, as desired.
\end{proof}

Note that this lemma is provable in finite set theory $\ZFC^{\neg\infty}$. It will be generalized by lemma \ref{Lemma.RealizingWellfoundedDigraphsAsSets}, which handles the general case of well-founded digraphs.

\begin{proof}[Proof of theorem \ref{Theorem.NonstandardModelsAreUniversal}]
Suppose that $M$ is a nonstandard model of \PA, and that $\<\HF^M,{\in^M}>$ is the corresponding nonstandard model of finite set theory
$\ZFC^{\neg\infty}$ built via the Ackerman coding in $M$. Consider the computable construction of the countable random $\Q$-graded digraph $\Gamma$, as carried out inside $M$. Let $k$ be a nonstandard integer of $M$, and consider the $k^\th$ stage of this construction $\Gamma_k^M$. Since the
standard stages of construction will be the same inside $M$ as externally, it follows that the standard part of $\Gamma_k^M$ is the actual countable random $\Q$-graded digraph $\Gamma$ itself. That is, $\Gamma$ is an induced subgraph of $\Gamma_k^M$. Since $M$ thinks $\Gamma_k$
is finite, we may apply lemma \ref{Lemma.FiniteDigraphsAsSets} inside $M$ to find a set $A\in\HF^M$ such that $M$ thinks
$\<\Gamma_k,{\toward}>\iso\<A,{\in^M}>$. Since $\Gamma$ is an induced subgraph of $\Gamma_k^M$, we may restrict this isomorphism (externally)
to find $\<\Gamma,{\toward}>\iso\<B,{\in^M}>$ for some $B\of \set{a\in\HF^M\st M\satisfies a\in A}$. Thus, we have realized the countable
random $\Q$-graded digraph as isomorphic to a collection of sets in $\HF^M$. But since $\Gamma$ is universal among all countable acyclic digraphs,
it follows that every countable set with an acyclic binary relation is isomorphic to an induced subgraph of $\Gamma$, and hence to a subcollection
of $\HF^M$ under $\in^M$, as desired.

Lastly, let me discuss the issue of weakening the theory \PA\ to $I\Delta_0$, which allows induction only for $\Delta_0$ assertions, and
even down to $IE_1$, which has induction only for bounded existential assertions. The reason is that every countable nonstandard model of
$IE_1$ has an initial segment that is a model of \PA, by a result of Paris \cite{Paris1984:OnTheStructureOfBoundedE1Induction(Cech)}, whose
proof relies on an earlier corresponding result for $I\Delta_0$, proved by McAloon \cite{McAloon1982:OnTheComplexityOfModelsOfArithmetic}.
So if we begin with a model $M\satisfies IE_1$, therefore, we may simply cut down to the $\PA$ initial segment $M_0\satisfies\PA$, where we
will find the universal structure by the argument of the previous paragraph. Thus, every model of an acyclic binary relation embeds into
$\<\HF^{M_0},{\in^{M_0}}>$, which is an initial segment of $\<\HF^M,{\in^M}>$, and so every model of an acyclic binary relation embeds into $\HF^M$,
as desired.
\end{proof}

\begin{corollary}
The countable models of set theory, up to isomorphism, have a universal object under the submodel relation. Namely, if $\<M,{\in^M}>$ is any $\omega$-nonstandard
model of set theory, then every countable model of set theory is isomorphic to a submodel of $\<M,{\in^M}>$.
\end{corollary}

\begin{proof}
Indeed, theorem \ref{Theorem.NonstandardModelsAreUniversal} shows that every countable model $\<N,E>$ of any binary relation embeds already
into the hereditary finite sets $\<\HF^M,{\in^M}>$ of $M$.
\end{proof}

I will show later that in fact {\it every} countable nonstandard model of set theory is universal in this way, even if the ill-foundedness appears only higher than $\omega$.

\section{$\ell$-graded digraphs and the finite-pattern property}

It will be convenient to consider grading notions other than $\Q$, aiming eventually at the well-founded acylic digraphs, which are precisely
the $\alpha$-graded digraphs for some ordinal $\alpha$, or in the case of proper classes, the $\Ord$-graded digraphs, as well as the $\No$-graded digraphs, using the surreal numbers $\No$. To assist with this
analysis, consider the general case of an arbitrary linear order $\ell$ and the $\ell$-graded digraphs $G$, which have a value assignment $v\mapsto
\alpha_v$ of the nodes $v\in G$ with elements $\alpha_v\in \ell$ in such a way that $v\toward w$ in $G$ implies $\alpha_v< \alpha_w$ in $\ell$.
Any such graph is acyclic, since a directed cycle in $G$ would give rise to a violation of asymmetry in $\ell$.

\begin{theorem}\label{Theorem.CountableRandomL-gradedDigraph}
For any countable linear order $\ell$, there is a countable homogeneous $\ell$-graded digraph that is universal for all countable $\ell$-graded
digraphs. This digraph is unique up to isomorphism (even when restricting universality to the finite $\ell$-graded digraphs), and has a
computable presentation if $\ell$ is computable.
\end{theorem}

\begin{proof}
The collection $\mathcal{Q}_\ell$ of all finite $\ell$-graded directed graphs has the hereditary property, the joint embedding property and the
amalgamation property, and consequently admits a \Fraisse\ limit, which is up to isomorphism the unique countable, homogeneous, $\ell$-graded
digraph of its age. This graph is universal for all countable $\ell$-graded directed graphs by essentially the same argument as in theorem
\ref{Theorem.TheCountableRandomQgradedDigraph}.

Alternatively, it may be easier simply to view $\ell$ as a suborder of the rational order $\Q$, and then consider the restriction $\Gamma\restrict\ell$ of the countable random $\Q$-graded digraph $\Gamma$ to the vertices having value in $\ell$. The subgraph $\Gamma\restrict\ell$ is $\ell$-graded and
inherits the homogeneity property from $\Gamma$, since any finite partial automorphism of $\Gamma\restrict\ell$ is also a finite partial
automorphism of $\Gamma$, which therefore extends to an automorphism of $\Gamma$, which respects the values and thus provides an
automorphism of $\Gamma\restrict\ell$. The digraph $\Gamma\restrict\ell$ remains universal for $\ell$-graded directed graphs, because any such
graph embeds into $\Gamma$ in a way that uses only the labels in $\ell$ and hence embeds into $\Gamma\restrict\ell$.

If $\ell$ is computable, then $\ell$ has a computable copy inside $\Q$, and using the computable presentation of $\Gamma$ from theorem
\ref{Theorem.TheCountableRandomQgradedDigraph}, we obtain a computable presentation $\Gamma\restrict\ell$, as desired.
\end{proof}

Let us call this graph the {\df countable random $\ell$-graded digraph}. There is also a finite-pattern property for these graphs analogous
to the $\omega$-categorical first-order characterization of the countable random graph, mentioned in observation \ref{Observation.CountableRandomDigraphIsCountableRandomGraph}.

\begin{definition}\label{Definition.FinitePatternProperty}\rm
An $\ell$-graded digraph $\Gamma$ satisfies the {\df finite-pattern property}, if for any disjoint finite sets of
vertices $A$, $B$ and $C$ and any $\alpha\in \ell$, such that every vertex in $A$ has value less than $\alpha$ and every vertex in $B$ has value greater than
$\alpha$, then there is a vertex $v$ in $\Gamma$ with value exactly $\alpha$ such that $a\toward v$
 and $v\toward b$ for every $a\in A$ and $b\in B$, but $v$ has no edges with any vertex in $C$.
\end{definition}

\begin{figure}[ht]
\begin{tikzpicture}[-left to]
  \draw (0,0) circle (1);
  \draw (0,4) circle (1);
  \draw (3,2) ellipse (1.3);
  \node[fill=white,anchor=north east] at (0,0) {$A$};
  \node[fill=white,anchor=south east] at (0,4) {$B$};
  \node[fill=white,anchor=west] at (3,2) {$C$};
  \node at (0,2) (v) {$\circ$};
  \node[anchor=south west] at (0,2) {$v$};
  \node[fill=white] at (-2,2) (alpha) {$\alpha$};
  \draw[dotted,-] (alpha) -- (-.5,2);
  \node at (-.4,.4) (a) {$\bullet$};
  \node at (.4,.1) (a') {$\bullet$};
  \node at (-.4,3.6) (b) {$\bullet$};
  \node at (.5,4.2) (b') {$\bullet$};
  \node at (2.5,1.5) {$\bullet$};
  \node at (2.8,2.2) {$\bullet$};
  \draw (a) -> (v);
  \draw (a') to (v);
  \draw (v) to (b);
  \draw (v) to (b');
\end{tikzpicture}
\caption{The finite-pattern property}
\end{figure}
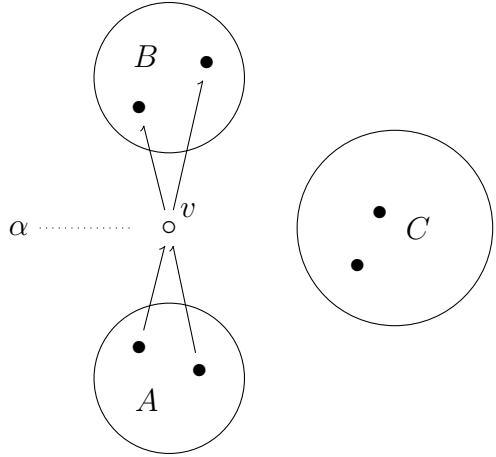

Equivalently, an $\ell$-graded digraph has the finite-pattern property just in case it is existentially closed in the class of all
$\ell$-graded digraphs.

\begin{theorem}\label{Theorem.FinitePatternPropertyIFF}
If\/ $\Gamma$ is a countable $\ell$-graded digraph for a countable linear order $\ell$, then the following are equivalent:
 \begin{enumerate}
  \item $\Gamma$ has the finite-pattern property.
  \item $\Gamma$ is isomorphic to the countable random $\ell$-graded digraph.
 \end{enumerate}
\end{theorem}

\begin{proof}
It is clear that the countable random $\ell$-graded digraph has the finite-pattern property, since each instance with particular $A$, $B$ and
$C$ is clearly realizable in a finite $\ell$-graded digraph, which must therefore embed into the countable random $\ell$-graded digraph by universality, thereby providing a realizing node $v$ by homogeneity.

Conversely, suppose that $\Gamma$ is a countable $\ell$-graded digraph with the finite-pattern property. It is clear by successive
applications of this property that every $\ell$-graded directed graph occurs as an induced subgraph of $\Gamma$. Now, one simply applies a
back-and-forth construction to realize that $\Gamma$ is isomorphic to the countable random $\ell$-graded digraph. At each stage, one applies
the finite-pattern property either in $\Gamma$ or in the countable random $\ell$-graded digraph to extend the isomorphism one more step.
\end{proof}

Although the random digraphs that we have constructed are homogenous as $\Q$-graded or as $\ell$-graded digraphs, as in observation \ref{Observation.NoHomogeneousDigraphWithoutGrading} they are
definitely not homogeneous merely as digraphs, without the extra structure provided by the value assignment. Meanwhile, if one should forget the values and the directionality of the edges, then the argument of observation \ref{Observation.CountableRandomDigraphIsCountableRandomGraph} shows that the graph underlying the countable random $\ell$-graded digraph, for any infinite linear order $\ell$, is isomorphic to the countable random graph.

\section{The surreal numbers and the hypnagogic digraph}

One may undertake an uncountable and even a class-sized analogue of the \Fraisse\ limit construction to produce a set-homogeneous graded class digraph that is universal for all graded digraphs, and assuming the global axiom of choice, it will also be universal for all proper class graded digraphs. In this section, I would like to describe how this graph also arises by a canonical construction intimately connected with the construction of the surreal numbers and what I call the hypnagogic digraph.

Let us first recall the standard construction of the class of surreal numbers, denoted $\No$. In this article, I am concerned with the surreal numbers only as a linear order; the ordered field structure and other structure on the surreals will not be relevant here.\footnote{Considerably pre-dating the surreal numbers, the study of what we now call the saturated linear orders was undertaken first by Hausdorff \cite{Hausdorff2002:GesammelteWerke}, who introduced the notation $\eta$ for the countable dense linear order and $\eta_\alpha$ for the saturated dense linear order of size $\aleph_\alpha$, when $\aleph_\alpha^{\smalllt\aleph_\alpha}=\aleph_\alpha$. In this notation, the surreal number line is simply $\eta_{\Ord}$.} To construct the surreal number line, one begins with nothing and then proceeds relentlessly, transfinitely, to fill all possible cuts in the order created so far. One defines the class $\No$ of all surreal numbers and their order in a simultaneous recursion, for which a surreal number is born at an ordinal stage if it is represented by a pair $\set{A\st B}$, where $A$ and $B$ are sets of previously-born surreal numbers, where every number in $A$ is smaller than every number in $B$. The idea is that $\set{A\st B}$ will be a surreal number filling the cut between $A$ and $B$, so that it is larger than every element of $A$ and smaller than every element of $B$. The very first surreal number to be created is $0=\set{\ \st\ }$, meaning by this notation that we have the empty set on each side, and then immediately afterwards the cut above $1=\set{0\st\ }$, meaning the singleton set on the left and the empty set on the right, and similarly the cut below $-1=\set{\ \st 0}$. Just  as with the rational numbers, however, the left-set/right-set $\set{A\st B}$ representation of the surreal numbers will not be unique, for different pairs of sets can determine the same cut, and so we will quotient by an underlying equivalence relation. Specifically, for surreal numbers  $x=\set{X_L\st X_R}$ and $y=\set{Y_L\st Y_R}$, we define $x\leq y$ if and only if no obvious obstacle prevents it, namely, there is no $x_L\in X_L$ with $y\leq x_L$ and no $y_R\in Y_R$ with $y_R\leq x$. Two numbers $x$ and $y$ are equivalent when $x\leq y\leq x$, and in this case we say that $x$ and $y$ are equal as surreal numbers. It is easy to prove that a surreal number $\set{A\st B}$ is equivalent to the first-born surreal number that is bigger than every element of $A$ and smaller than every element of $B$.

One sees immediately the similarity in the construction of the surreal numbers and the \Fraisse\ limit construction, and indeed, the surreal numbers are simply the proper class \Fraisse\ limit of the class of all finite linear orders. The surreal numbers are set-homogeneous and universal for all proper class linear orders (provided global choice holds, but one can construct them and get universality for set linear orders just in \ZFC). In this sense, the surreal number line is for proper class linear orders what the rational number line $\Q$ is for the countable linear orders. What I aim to do in this section is to continue this analogy by finding the corresponding proper class analogue of the countable random $\Q$-graded digraph. The answer is what I call the {\df hypnagogic digraph}, denoted $\Hg$. To explain the terminology, the hypnagogic state is the dream-like, sometimes hallucinatory state at the boundary between wakefulness and sleep, and its usage here is meant to evoke both the universal property of the graph as well as its homogeneity: peering into the hypnagogic digraph, one may see all the various dancing visions of any given graded digraph.

\begin{theorem}\label{Theorem.HynagogicDigraph}
There is surreal-numbers-graded class digraph $\Hg$, the hypnagogic digraph, such that:
\begin{enumerate}
 \item $\Hg$ is universal for all graded digraphs using any linear order, as a suborder of $\No$.
 \item $\Hg$ is set-homogeneous, meaning that any isomorphism of two induced graded subgraphs extends to an automorphism of $\Hg$.
 \item If global choice holds, then $\Hg$ is universal for all proper class graded digraphs, using any class linear order as a suborder of $\No$, and furthermore is uniquely determined up to isomorphism by (1) and (2).
 \item $\Hg$ exhibits the set-pattern property.
\end{enumerate}
\end{theorem}

\begin{proof}
Just as for the countable random $\Q$-graded digraph in theorem \ref{Theorem.TheCountableRandomQgradedDigraph}, there are numerous constructions of this highly canonical object. For example, the hypnagogic digraph is the \Fraisse\ limit of the class of all finite $\No$-graded digraphs. It is simply the $\Ord$-saturated $\No$-graded class digraph. But meanwhile, there is also a highly canonical construction of the hypnagogic digraph, closely connected with the construction of the surreal numbers, which I prefer to take as the official definition of the class $\Hg$. The idea is simply to use the left-set/right-set representation $\set{A\st B}$ of the surreal numbers, {\it but do not quotient by the equivalence relation!} Rather, the hypnagogic digraph is the term algebra of the terms representing surreal numbers. That is, the nodes in $\Hg$ are all the various representations $\set{A\st B}$ of surreal numbers, counted as different nodes even if they are equal as surreal numbers. For a given node $v=\set{A\st B}$, I place edges $a\toward v$ and $v\toward b$, for every $a\in A$ and $b\in B$, and define the value of $v$ to be simply the surreal number that it represents. Thus, the meaning of $\set{A\mid B}$ as a node in $\Hg$ is that it is pointed at from every $a\in A$; it  points at every $b\in B$; and its value is $\set{A\mid B}$ as a surreal number. The hypnagogic digraph is thereby the result of systematically adding nodes with a specified edge connectivity with the already constructed nodes. Under this description, it is clear that $\Hg$ satisfies the set-pattern property, which in turn implies the universality property and thus the set-homogeneity by a class-length back and forth property. (Note: one need not use the global choice principle for the set-homogeneity of $\Hg$, since any automorphism of the nodes up to a given birthday induces a canonical automorphism of the cuts created at that birthday.) One gets universality for all class $\ell$-graded digraphs, using any linear order $\ell$, simply because under the global choice principle every linear order embeds as a suborder of $\No$, and then one can systematically build an embedding of the given $\ell$-graded digraph by mapping to a node representing the desired value and connectivity. Uniqueness follows similarly by a class-length back-and-forth construction.
\end{proof}

One way to view the construction of the hypnagogic digraph is that when a surreal number $v=\set{A\st B}$ is born, the meaning of this term is that
$v$ is definitely larger than every element of $A$ and smaller than every element of $B$. The hypnagogic digraph preserves this information by adding digraph edges exactly to record it. In this way, the hypnagogic digraph underlies the surreal numbers.

To further support the idea that the hypnagogic digraph $\Hg$ is the proper class analogue of the countable random $\Q$-graded digraph, observe that the hypnagogic digraph $\Hg^M$ as computed inside a countable model $M\satisfies\ZFC$ is externally isomorphic to the countable random $\Q$-graded digraph. This is because the surreal numbers $\No^M$ as computed in $M$ are a countable dense linear order without endpoints, and thus isomorphic to the rational line $\Q$, and $\Hg^M$ is thus essentially a countable $\Q$-graded digraph with the finite-pattern property, which is isomorphic to the countable random $\Q$-graded digraph by theorem \ref{Theorem.FinitePatternPropertyIFF}.

The surreal number line $\No$ is universal for all class linear orders. For any such suborder $\ell\of\No$, the {\df hypnagogic $\ell$-graded digraph} $\Hg\restrict\ell$ is formed simply by restricting the hypnagogic digraph to the nodes with value in $\ell$.

\begin{theorem}\label{Theorem.HypnagogicLgradedDigraph}
For any class linear order $\ell\of\No$, the hypnagogic $\ell$-graded digraph $\Hg\restrict\ell$ is set-homogeneous and universal for all $\ell$-graded digraphs (including class digraphs if global choice holds). Furthermore, it is uniquely determined by these properties.
\end{theorem}

\begin{proof}
Clearly $\Hg\restrict\ell$ is an $\ell$-graded digraph, and it inherits the set-homogeneity of $\Hg$, since any isomorphism of induced subgraphs of $\Hg\restrict\ell$ is also an isomorphism of induced subgraphs of $\Hg$, which therefore extends to an automorphism of all of $\Hg$, which provides an automorphism of $\Hg\restrict\ell$ since these automorphisms respect the values of nodes. Furthermore, $\Hg\restrict\ell$ is universal for all $\ell$-graded digraphs, since $\Hg$ is, by embeddings that respect the values of nodes. Uniqueness follows (assuming the global choice principle) by the usual back-and-forth argument.
\end{proof}

In particular, since the ordinals are a suborder of the surreal numbers, $\Ord\of\No$, we have the hypnagogic $\Ord$-graded digraph $\Hg\restrict\Ord$, which is set-homogeneous and universal for all $\Ord$-graded digraphs. One can view the hypnagogic $\Ord$-graded digraph as constructed in a grand recursion: at each stage, for any ordinal $\alpha$ up to that stage and any sets $A$ and $B$ of previously constructed nodes, with those in $A$ having value less than $\alpha$ and those in $B$ above, one forms a new node $v$ with value $\alpha$ and creates edges $a\toward v$ and $v\toward b$ for every $a\in A$ and $b\in B$.

Note that if $M$ is a countable transitive model of set theory, with height $\lambda=\Ord^M$, then the hypnagogic $\Ord$-graded digraph $(\Hg\restrict\Ord)^M$ as constructed in $M$ is isomorphic by theorem \ref{Theorem.FinitePatternPropertyIFF} to the countable random $\lambda$-graded digraph.

\section{Realizing well-founded digraphs as sets}\label{Section.RealizingWfDigraphsAsSets}

The following basic lemma, a modification of the Mostowski collapse construction, shows how to realize well-founded acyclic digraphs as sets. Similar modified Mostowski collapse maps have arisen for diverse purposes, such as in the non-well-founded set theory of Aczel \cite{Aczel1988:Non-well-foundedSets} and in the finite set theory of Kirby \cite{Kirby2010:SubstandardModelsOfFiniteSetTheory} and elsewhere.

\begin{lemma}\label{Lemma.RealizingWellfoundedDigraphsAsSets}
Every well-founded digraph $(G,\toward)$ is isomorphic to a set $(A,{\in})$ under the set-membership relation $\in$.
Furthermore,
\begin{enumerate}
 \item If $G$ is finite, then there is a hereditarily finite such set $A$.
 \item If $G$ is countable and $\lambda$-graded for some ordinal $\lambda$, then $A$ can be found with $A\of V_{\omega+\lambda}$. In particular, if $\omega^2\leq\lambda$, then there is such $A\of V_\lambda$.
 \item If $G$ is $\lambda$-graded and $|G|\leq|V_\beta|$, then $A$ can be found with $A\of V_{\beta+2+\lambda}$. In particular, if
     $\beta^2\leq\lambda$, then there is such $A\of V_\lambda$.
 \item If $G$ is countable and $\lambda$-graded, with $\lambda$ infinite, then $A$ can be found with $A\of L_{\lambda+\lambda}$.
\end{enumerate}
\end{lemma}

\begin{proof}
Suppose that $(G,\toward)$ is a well-founded digraph. Define
 $$\pi(x)=\set{\pi(y)\st y\toward x}\union\set{\set{\emptyset,x}},$$
and let $A=\set{\pi(x)\st x\in G}$. This definition is well-defined by recursion on the well-founded relation $\toward$, since the value of $\pi(x)$ is determined by the values of $\pi(y)$ for earlier $y\toward x$. I claim that
it is an isomorphism of $(G,\toward)$ with $(A,{\in})$. It is essentially similar, of course, to the Mostowski collapse, but modified by
the inclusion of the extra element $\set{\emptyset,x}$, which is added in order to distinguish the nodes sufficiently. (In the case that
the original graph is extensional, then the nodes are already distinguished by their predecessors and this modification is unnecessary, for
the Mostowski collapse is already an isomorphism.) It is clear from the definition that $y\toward x$ implies $\pi(y)\in\pi(x)$. For the
converse direction, note first that no $\pi(x)$ is empty and since $\set{\emptyset,x}$ is nonempty, it follows also that
$\emptyset\notin\pi(y)$ for any $y$. In particular, $\pi(y)\neq\set{\emptyset,x}$ for any $x$ and $y$. Thus, if $\pi(y)\in\pi(x)$, it must
be that $\pi(y)=\pi(z)$ for some $z\toward x$. Since $\set{\emptyset,y}\in \pi(y)$ and $\set{\emptyset,z}\in\pi(z)$, but these sets are not
$\pi(u)$ for any $u$, it follows that $\set{\emptyset,y}=\set{\emptyset,z}$ and consequently $y=z$ and thus actually $y\toward x$. So I have
established for all $x,y\in G$ that
$$y\toward x\qquad\text{ if and only if }\qquad\pi(y)\in\pi(x),$$
and this is precisely what it means for $\pi$ to be an isomorphism of $(G,\toward)$ with $(A,{\in})$. This establishes the basic claim of the lemma.

I now consider the further claims by analyzing the nature of $A$. If $G$ is finite, then we may assume that the underlying nodes of $G$ are
natural numbers (or some other hereditary finite sets), and in particular, the sets $\set{\emptyset,x}$ arising in the definition of $\pi$
are all hereditarily finite. It follows inductively that every $\pi(x)$ is hereditarily finite, and so $A$ is a finite set of hereditarily
finite sets and thus hereditarily finite itself, establishing (1). Similarly, for (2), if $G$ is countable, then we may again assume that
the nodes of $G$ are natural numbers and consider the isomorphism $\pi$ as defined above. By induction on values in the grading, it follows
for any node $x$ with value $\alpha$ that $\pi(x)\in V_{\omega+\alpha}$, since $\set{\emptyset,x}\in V_\omega$ and $\pi(x)$ consists
otherwise of $\pi(y)$, where $y$ has some value $\beta<\alpha$ and hence by induction $\pi(y)\in V_{\omega+\beta}$. So $A\of
V_{\omega+\lambda}$. If $\omega^2\leq\lambda$, then $\omega+\lambda=\lambda$, and so in this case $A\of V_\lambda$.

A similar argument works in the case of (3). Assume that $G$ is $\lambda$-graded and the nodes of $G$ come from $V_\beta$. In this case,
$\set{\emptyset,x}\in V_{\beta+1}$ at worst, and inductively $\pi(x)\in V_{\beta+2+\alpha}$ when $x$ has value $\alpha$. And so $A\of
V_{\beta+2+\lambda}$. If $\beta^2\leq\lambda$, then $\beta+2+\lambda=\lambda$, and so we have $A\of V_\lambda$, as desired.

Assertion (4) is subtle. As in theorem \ref{Theorem.HypnagogicLgradedDigraph}, let $\Lambda=(\Hg\restrict\lambda)^{L_\lambda}$ be the hypnagogic $\lambda$-graded digraph as defined in $L_\lambda$. Since this graph is universal for all countable $\lambda$-graded digraphs, the original digraph $G$ embeds into $\Lambda$, and so it will suffice just to handle $\Lambda$. This is a definable class in $L_\lambda$. Although the edge relation $\toward$ of $\Lambda$ is definable in $L_\lambda$, we will not be able to carry out the modified Mostowski collapse of $\Lambda$ inside $L_\lambda$, because $\Lambda$ has $\lambda$ many nodes of each given value, and indeed, every node in $\Lambda$ has $\lambda$ many predecessors there. Thus, the sets arising at every step of the
modified Mostowski collapse will have size $\lambda$, and consequently will not necessarily be elements of $L_\lambda$. Another way to say
it is that the edge relation $\toward$ of $\Lambda$ is not set-like in $L_\lambda$, and this is why if we want to use this graph directly, we have to build sets on top of $L_\lambda$, stretching up to $L_{\lambda+\lambda}$. Specifically, let $\pi:\Lambda\to A$ be the modified Mostowski collapse as
defined above. I claim that $A\of L_{\lambda+\lambda}$, by arguing that any node $x\in\Lambda$ of value $\alpha$ has $\pi(x)\in
L_{\lambda+1+\alpha+1}$. This is true if $x$ has value $0$, since $x\in L_\lambda$ and $\pi(x)=\set{\set{\emptyset,x}}\in L_{\lambda+2}$ at
worst (although if $\lambda$ is a limit ordinal, one achieves $L_\lambda$ here). More generally, if the claim is true for all values below $\alpha$ and $x$ has value $\alpha$, then $\pi(x)=\set{\pi(y)\st y\toward
x}\union\set{\set{\emptyset,x}}$, where for $y\toward x$ we have $\pi(y)\in L_{\lambda+1+\beta}$ for some $\beta<\alpha$, and so $\pi(x)$
is a subset of $L_{\lambda+1+\alpha}$, which is definable from $\Lambda,x\in L_\lambda$, and so $\pi(x)\in L_{\lambda+1+\alpha+1}$. Thus,
altogether, $A\of L_{\lambda+\lambda}$. So we have found a set $A\of L_{\lambda+\lambda}$ such that $(A,{\in})$ is the countable random
$\lambda$-graded digraph. It now follows by the universality property of this digraph identified in theorem
\ref{Theorem.CountableRandomL-gradedDigraph} that any countable $\lambda$-graded digraph is isomorphic to a subset of this particular $A$,
establishing (4).
\end{proof}

It will follow as a consequence of the main theorem, theorem \ref{Theorem.MisSubmodelOfL^M}, that we may actually improve statements (2) and (4) of lemma \ref{Lemma.RealizingWellfoundedDigraphsAsSets} to the following assertion:
\begin{enumerate}
 \item[(5)] If $G$ is countable and $\lambda$-graded with $\lambda$ infinite, then $\<G,{\toward}>\iso\<A,{\in}>$ for some $A\of L_\lambda$.
\end{enumerate}
Indeed, this will be the key point of the proof of the main theorem.

The proof of lemma \ref{Lemma.RealizingWellfoundedDigraphsAsSets} also establishes the following, illustrating that there is an enormous difference between the embedding concept and that of an elementary embedding or even a $\Sigma_1$-elementary embedding in the context of models of set theory. The Kunen inconsistency \cite{Kunen1971:ElementaryEmbeddingsAndInfinitaryCombinatorics} (see also \cite{HamkinsKirmayerPerlmutter2012:GeneralizationsOfKunenInconsistency}) shows that there can be no $\Sigma_1$-elementary embedding from the universe to itself.

\begin{theorem}\label{Theorem.j:VtoV}
 There is a definable nontrivial embedding $j:V\to V$ from the universe to itself.
\end{theorem}

\begin{proof}
Let $j(x)=\set{j(y)\st y\in x}\union\set{\set{\emptyset,x}}$, which is well-defined by recursion on the well-founded relation $\in$. The proof of lemma \ref{Lemma.RealizingWellfoundedDigraphsAsSets} shows that $y\in x\iff j(y)\in j(x)$, and so this is an embedding from $V$ to $V$, and it is clearly nontrivial.
\end{proof}

The same idea produces embeddings $j:L\to L$, regardless of the existence of $0^\sharp$, simply by applying theorem \ref{Theorem.j:VtoV} inside $L$.

\section{Warming up to the main theorems}

In this section I explain how lemma \ref{Lemma.RealizingWellfoundedDigraphsAsSets} can be used to prove several approximations to main theorem \ref{MainTheorem.MsubmodelL^M}, which I give here as a warm-up, since the proof of main theorem \ref{MainTheorem.MsubmodelL^M} will introduce several complications.

\begin{proposition}\label{Proposition.MisSubmodelOfL_lambda+lambda}
 If $M$ is any countable transitive model of set theory and $\lambda=\Ord^M$, then $\<M,{\in}>$ is isomorphic to a submodel of $\<L_{\lambda+\lambda},{\in}>$.
\end{proposition}

\begin{proof}
This proposition is a quick corollary to statement (4) of lemma \ref{Lemma.RealizingWellfoundedDigraphsAsSets}, since if we view $\<M,{\in^M}>$ as a $\lambda$-graded digraph as in observation \ref{Observation.ModelOfSetTheoryAsOrdGradedDigraph}, then lemma \ref{Lemma.RealizingWellfoundedDigraphsAsSets} statement (4) says that it is isomorphic to a submodel of $\<L_{\lambda+\lambda},{\in}>$, as desired.
%
%
%
\end{proof}

The need to go to $L_{\lambda+\lambda}$ in proposition \ref{Proposition.MisSubmodelOfL_lambda+lambda} and in lemma \ref{Lemma.RealizingWellfoundedDigraphsAsSets} statement (4) is directly connected with the fact that the hypnagogic $\lambda$-graded digraph $\Lambda$ as computed in $L_\lambda$ is not set-like in $L_\lambda$, since every node has $\lambda$ many predecessors of each smaller value. Perhaps one might hope to overcome this difficulty by finding a $\lambda$-graded digraph in $L_\lambda$ that was both set-like and had the finite-pattern property, for in this case, it would be universal for all $\lambda$-graded binary relations and we would be able to perform the modified Mostowski collapse of lemma \ref{Lemma.RealizingWellfoundedDigraphsAsSets} inside $L_\lambda$ itself. This would prove that $M$ is isomorphic to a submodel of $L_\lambda=L^M$, which is the main goal here. Unfortunately, however, the following result shows that there is no such digraph in any model of \ZF.

\begin{observation}\label{Observaton.NoSetLikeOrdGradedDigraphWithFPP}
No model of \ZF\ has an $\Ord$-graded digraph class $\Gamma$ with the finite-pattern property, such that $\Gamma$ is set-like.
\end{observation}

\begin{proof}
Suppose that $\Gamma$ is a set-like $\Ord$-graded digraph class with the finite-pattern property. Consider any fixed node $p$ with nonzero
value, and let $A_p$ be the set of predecessors of $p$ in $\Gamma$, which is a set precisely because we assumed that $\Gamma$ is set-like.
Since $A_p$ has only a set number of subsets, but there are a proper class of higher levels, it must be that there are two distinct nodes
$q,r$ with value above $p$, but with the same pattern of predecessors on $A_p$. That is, for $v\in A_p$, we have $v\toward q$ if and only if $v\toward r$. This violates the finite-pattern property, which would require the existence of a node $v$ with $v\toward p$ and $v\toward q$ but $v\perp r$. Such a vertex $v$ would reveal that $q$ and $r$ have different predecessors below $p$, contrary to assumption.
\end{proof}

The complications introduced into the proof of the main theorem, using the surrogate parent concept, are aimed specifically at overcoming this difficulty. Meanwhile, we can also overcome the difficulty by relaxing to a weaker but still commonly considered theory. Let ``$V=H_{\kappa^+}$'' be the theory asserting that $\kappa$ is the largest cardinal and that every set has hereditary size at most $\kappa$. For any cardinal $\kappa$, the collection $H_{\kappa^+}$ is a model of $\ZFC^-+V=H_{\kappa^+}$, and these structures and their elementary substructures are very commonly considered in set theory.

\begin{theorem}\label{Theorem.LambdaInModelOfV=H_kappa+}
Suppose $M$ is a countable transitive model of $\ZFC^-+V=H_{\kappa^+}$, where $\kappa^\ltkappa=\kappa$ in $M$ and $\lambda=\Ord^M$. Then there is a $\lambda$-graded digraph $\Lambda\of M$, which is set-like in $M$ and which obeys the $\ltkappa$-pattern property in $M$.
\end{theorem}

\begin{proof}
I shall modify the construction of the hypnagogic $\lambda$-graded digraph so that it will become set-like, by restricting it to have $\kappa$ many nodes of any given value. Specifically, for each $\beta<\lambda$ of size $\kappa$, consider the construction in $M$ of the $\beta$-graded hypnagogic digraph as the \Fraisse\ limit of the finite $\beta$-graded digraphs. This graph is an element of $M$ and has the $\ltkappa$-pattern property in $M$. Furthermore, for any $\beta<\gamma<\lambda$ and any copy of the $\beta$-graded hypnagogic digraph $\Lambda_\beta$, it may be extended to the hypnagogic $\gamma$-graded digraph $\Lambda_\gamma$ as constructed in $M$, and furthermore extended in such a way that $\Lambda_\beta$ is precisely $\Lambda_\gamma\restrict\beta$. This is simply because $\Lambda_\gamma\restrict\beta$ is $\kappa$-homogeneous and has the $\ltkappa$-pattern property, and hence is isomorphic to $\Lambda_\beta$. Continuing the main argument, now, I select an increasing sequence $\<\lambda_n\st n<\omega>$ cofinal in $\lambda$ and form the limit graph $\Lambda=\Union_n \Lambda_{\lambda_n}$, where the $\Lambda_n$ form a coherent sequence of induced subgraphs, with $\Lambda_{\lambda_n}=\Lambda\restrict\lambda_n$. The graph $\Lambda$ is set-like in $M$, because every node lives in some $\Lambda_n$, which is an element of $M$, and all its children are also in $\Lambda_n$. Furthermore, $\Lambda$ is $\lambda$-graded and has the finite-pattern property. Moreover, each $\Lambda_n$ and consequently also $\Lambda$ has the $\ltkappa$-pattern property with respect to $M$, since every subset of $\Lambda$ that is in $M$ is contained in some $\Lambda_n$.
\end{proof}

Note that the final construction of $\Lambda$ given in theorem \ref{Theorem.LambdaInModelOfV=H_kappa+} was external to $M$, using the sequence $\<\lambda_n\st n<\omega>$, although the graph was built as the union of initial segments that were elements of $M$ and consequently $\Lambda$ is at least amenable to $M$. But there seems little reason to expect for this construction that $\<M,\in,\Lambda>$ satisfies $\ZFC^-(\Lambda)$. Notice also that since $\lambda$ is countable, the digraph $\Lambda$ of theorem \ref{Theorem.LambdaInModelOfV=H_kappa+} is exactly the countable random $\lambda$-graded digraph. The important part of the theorem is that this digraph is set-like in $M$, which allows us to perform the modified Mostowski collapse of lemma \ref{Lemma.RealizingWellfoundedDigraphsAsSets}.

\begin{corollary}
 If $M$ is a countable transitive model of the theory $\ZFC^-+V=H_{\kappa^+}$ and $\kappa^\ltkappa=\kappa$ in $M$ and $\lambda=\Ord^M$, then there is a submodel $A\of M$ such that $\<A,{\in}>$
 is the countable random $\lambda$-graded digraph. Consequently, every countable $\lambda$-graded digraph is isomorphic to a submodel of $\<A,{\in}>$. In particular, every countable transitive model of set theory of height at most $\lambda$ is isomorphic to a submodel of $\<M,{\in}^M>$.
\end{corollary}

\begin{proof}
By theorem \ref{Theorem.LambdaInModelOfV=H_kappa+}, there is a set-like $\lambda$-graded digraph $\Lambda$ with the finite-pattern property, whose initial segments are sets in $M$. From this it follows that the modified Mostowski collapse of $\Lambda$ as defined in lemma \ref{Lemma.RealizingWellfoundedDigraphsAsSets} is an isomorphism of $\<\Lambda,{\toward}>\iso\<A,{\in}>$ for some $A\of M$. Since $\Lambda$ has the finite-pattern property, it and hence also $\<A,\in>$ are the countable random $\lambda$-graded digraph, which is universal for all countable $\lambda$-graded digraphs. Thus, every countable $\lambda$-graded digraph is isomorphic to a submodel of $\<A,\in>$ and hence of $\<M,\in>$. By observation \ref{Observation.ModelOfSetTheoryAsOrdGradedDigraph}, this includes all countable models of set theory of well-founded height at most $\lambda$.
\end{proof}

The key obstacle identified in the proof of observation \ref{Observaton.NoSetLikeOrdGradedDigraphWithFPP} is the existence of more ordinals
than the size of any given power set. The following theorem shows that this obstacle is the only obstacle, for if we are willing to give up
the power set by adding many Cohen reals, then we can find the desired universal digraph.

\begin{theorem}\label{Theorem.Add(omega,Ord)CreatesSetLikeDigraph}
Assume $V\satisfies\ZFC$. If $G\of\Add(\omega,\Ord)$ is $V$-generic for the forcing to add $\Ord$ many Cohen reals, then $V[G]\satisfies\ZFC^-$ has a submodel $A\of V[G]$ such that $\<A,\in>$ is an $\Ord$-graded digraph with the finite-pattern property.
\end{theorem}

\begin{proof}
The idea is to build a set-like $\Ord$-graded digraph $\Lambda$ inside $V[G]$. The nodes of $\Lambda$ will be precisely the elements of $\Ord\times\omega$, and $(\alpha,n)$ will have value $\alpha$. What remains is to specify the edges, which will be done generically. Let $\P$ be the class partial order consisting of a finite digraph $G$, whose vertices are contained in $\Ord\times\omega$, ordered by the (reverse) induced subgraph relation $G\leq H\iff H\of G$, so that as usual, stronger conditions are lower in the partial order.

I claim that this partial order is isomorphic to the forcing $\Add(\omega,\Ord)$ to add $\Ord$ many Cohen reals, which is the same as the
forcing to add a generic class function $g:\Ord\to 2$ by finite conditions. One can see this simply by observing that the forcing $\P$ is
deciding generically and independently for each pair of nodes whether to place an edge $(\alpha,n)\toward (\beta,k)$, where $\alpha<\beta$,
or not. Thus, $\P$ is adding, by finite support, a generic function that decides of each possible edge, whether it is there or not. Thus,
this forcing resonates with the random characterization of the countable random digraphs we saw earlier, but using genericity in place of
randomness.

Suppose that $\Lambda$ is the resulting $V$-generic $\Ord$-grade digraph. It is clear that we will achieve the finite-pattern property,
since if $A$, $B$ and $C$ are finite subsets of the nodes and $\alpha$ is an ordinal such that every node in $A$ has value less than
$\alpha$ and every node in $B$ has value larger than $\alpha$, then it is a dense requirement that there is a node with value $\alpha$ that
is above all nodes in $A$, below all nodes in $B$ and with no edges in $C$. The reason is that any condition mentions only finitely many
additional nodes, and there will be an unmentioned node on level $\alpha$ that we may simply add in so as to satisfy the desired pattern.

Finally, the graph $\Lambda$ is set-like in $V[G]$, because the predecessors of $(\alpha,n)$ in $\Lambda$ are amongst $\alpha\times\omega$,
which is a set. It follows that we may perform the modified Mostowski collapse of $\Lambda$ as in lemma \ref{Lemma.RealizingWellfoundedDigraphsAsSets}, and find $A\of L[G]$ such that $\<\Lambda,{\toward}>\iso\<A,\in>$.
\end{proof}

\begin{corollary}
 If $M$ is a countable transitive model of set theory of height $\lambda$,  then $M$ is isomorphic to a submodel of $\<L_\lambda[G],\in>$, a forcing extension of $L_\lambda$ obtained by adding $\lambda$ many Cohen reals.
\end{corollary}

\begin{proof}
By theorem \ref{Theorem.Add(omega,Ord)CreatesSetLikeDigraph}, there is $A\of L[G]$ such that $\<A,\in>$ is $\lambda$-graded and has the finite-pattern property. By theorem \ref{Theorem.FinitePatternPropertyIFF}, it follows that $\<A,\in>$ is the countable random $\lambda$-graded digraph, which is universal for all countable $\lambda$-graded digraphs. Since $\<M,\in>$ is a $\lambda$-graded digraph by observation \ref{Observation.ModelOfSetTheoryAsOrdGradedDigraph}, it follows that $\<M,\in>$ is isomorphic to a submodel of $\<A,\in>$ and hence to a submodel of $\<L[G],\in>$.
\end{proof}

This corollary will be improved by the main theorem, but it is already surprising, because the model $M$ may have many large cardinals or other complex objects such as $0^\sharp$ that are fundamentally incompatible with $V=L$ or $V=L[G]$. Nevertheless, the main theorem will omit the need for $G$ and obtain the even more surprising result that $M$ actually embeds as a submodel of $L_\lambda$ itself.

\section{Proving the main theorems}

Finally, I am ready to prove the main theorems \ref{MainTheorem.MsubmodelL^M} and \ref{MainTheorem.CountableModelsLinearlyOrdered}.

\begin{theorem}[Main Theorem 1]\label{Theorem.MisSubmodelOfL^M}
Every countable model of set theory $\<M,{\in^M}>$ is isomorphic to a submodel of its own constructible universe $\<L^M,{\in^M}>$. In other words, there is an embedding $$j:\<M,{\in^M}>\to\<L^M,{\in^M}>.$$
\end{theorem}

\begin{proof}
Suppose that $M$ is a countable model of \ZFC, not necessarily transitive or well-founded. Let $\lambda_n\in\Ord^M$ be an increasing cofinal sequence of limit ordinals from $M$. By the reflection theorem, I may assume without loss of generality that $V_{\lambda_n}^M\prec_{\Sigma_n} M$, although the argument will not really use much of this. Let $\Gamma_n$ be the hypnagogic $(\lambda_n+1)$-graded digraph as constructed in $L_{\lambda_{n+1}}^M$, running the construction for $\lambda_{n+1}$ many birthdays. Thus, each value level of $\Gamma_n$ has $\lambda_{n+1}$ many elements, and $\Gamma_n$ has a top layer consisting of the nodes with value $\lambda_n$. Furthermore, each $\Gamma_n$ has the finite-pattern property, each $\Gamma_n$ is an induced subgraph of $\Gamma_{n+1}$ and the union $\Gamma=\Union_n\Gamma_n$ of the chain is exactly the hypnagogic $\Ord$-graded digraph of $L^M$. Although each $\Gamma_n$ is a set in $L^M$, the union digraph $\Gamma$ is not set-like in $L^M$, since every node in $\Gamma_n$ gains $\lambda_{n+2}$ many new children in $\Gamma_{n+1}$.

I shall presently use the graphs $\Gamma_n$ in order to construct a somewhat more elaborate graph, which I call the surrogate digraph, which will be set-like and which will exhibit a kind of finite-pattern property that will be sufficient for universality. Define that $\<v_0,\ldots,v_n>$ is a {\df surrogate sequence} of its final node $v_n$, if $v_n\in\Gamma_n$, the value of $v_n$ is at least $\sup_{k<n}\lambda_k$ and each $v_k$ for $k<n$ is a node  in $\Gamma_k$ of value $\lambda_k$.
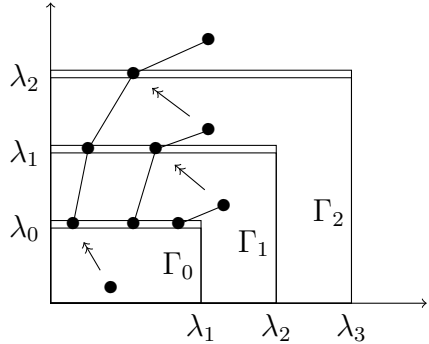
\begin{figure}[hc]
\begin{tikzpicture}
\draw[->] (0,0) to (0,4);
\draw[->] (0,0) to (5,0);
\draw (0,0) rectangle (2,1);
\draw (0,0) rectangle (2,1.1);
\draw (0,0) rectangle (3,2);
\draw (0,0) rectangle (3,2.1);
\draw (0,0) rectangle (4,3);
\draw (0,0) rectangle (4,3.1);
\node at (1.7,.5) {$\Gamma_0$};
\node at (2.7,.8) {$\Gamma_1$};
\node at (3.7,1.2) {$\Gamma_2$};
\node[anchor=east] at (0,1) {$\lambda_0$};
\node[anchor=east] at (0,2) {$\lambda_1$};
\node[anchor=east] at (0,3) {$\lambda_2$};
\node[anchor=north] at (2,0) {$\lambda_1$};
\node[anchor=north] at (3,0) {$\lambda_2$};
\node[anchor=north] at (4,0) {$\lambda_3$};
\draw[bend right] (.3,1.05) --(.5,2.05)--(1.1,3.05)--(2.1,3.5);
\node at (.3,1.05) (f) {$\bullet$};
\node at (.5,2.05) (h) {$\bullet$};
\node at (1.1,3.05) (b) {$\bullet$};
\node at (2.1,3.5) {$\bullet$};
\draw (1.1,1.05) -- (1.4,2.05) -- (2.1,2.3);
\node at (1.1,1.05) {$\bullet$};
\node at (1.4,2.05) (d) {$\bullet$};
\node at (2.1,2.3) (a) {$\bullet$};
\draw (1.7,1.05) -- (2.3,1.3);
\node at (1.7,1.05) {$\bullet$};
\node at (2.3,1.3) (c) {$\bullet$};
\draw[->>] (a) to (b);
\draw[->>] (c) to (d);
\node at (.8,.2) (e) {$\bullet$};
\draw[->>] (e) to (f);
\end{tikzpicture}
\caption{Surrogate sequences in the surrogate digraph}
\end{figure}
For surrogate sequences $v=\<v_0,\ldots,v_n>$ and $w=\<w_0,\ldots,w_m>$, define the surrogate edge relation $w\ttoward v$ to hold if and only if $m\leq n$ and $w_m\toward v_m$. The idea here is to split the parent and child roles in $\Gamma$, so that only the final node $w_m$ of a surrogate sequence $\<w_0,\ldots,w_m>$ acts as a child, while the prior nodes $w_k$ act as surrogate parents, as far as gaining new children below value $\lambda_k$ is concerned. It may be helpful to think of the relation that a set $x$ has with its projections $x\intersect V_\alpha$ for $\alpha$ of much smaller von Neumann rank than $x$, for we can tell if $y\in x$ or not, when $y$ has rank less than $\alpha$, merely by knowing whether or not $y\in x\intersect V_\alpha$. In our final analysis, the surrogates $v_k$ will be something like (but not exactly like) these projections, to capture the elements of rank less than $\lambda_k$, and to do so without leaving $\Gamma_k$. This surrogate maneuver will allow us to surmount the obstacle of observation \ref{Observaton.NoSetLikeOrdGradedDigraphWithFPP}.

Let $\Theta_n$ be the $(\lambda_n+1)$-graded digraph consisting of all surrogate sequences of length at most $n$, using $\ttoward$ as the edge relation, and giving the surrogate sequence the value of its final node in $\Gamma$. I shall refer to the union digraph $\Theta=\Union_n\Theta_n$ as the $\Ord^M$-graded {\df surrogate} digraph arising from $\<\lambda_n\st n<\omega>$. Because we can define the sequence $\<\lambda_n\st n<\omega>$ from $\Theta$, we cannot expect that it is a class in $M$. Nevertheless, each $\Theta_n$ is a set in $M$ and an induced subgraph of $\Theta$, the restriction of $\Theta$ to surrogate sequences having value at most $\lambda_n$. It follows that the full surrogate  digraph $\Theta$ is an amenable class over $M$ and indeed over $L^M$. In particular, $\Theta$ is set-like in $L^M$, since any particular surrogate sequence $\vec w$ lies in some $\Theta_n$, which is a set in $L^M$ and is the $(\lambda_n+1)$-initial segment of $\Theta$.

Let's define that two subsets of $\Theta$ are {\df completely disjoint} if they are disjoint and furthermore any two surrogate sequences from either of them have no nodes in common. That is, not only are they disjoint as sets of sequences, but the individual nodes on those sequences do not recur.

\begin{sublemma}
The surrogate digraph $\Theta$ enjoys the {\df surrogate finite-pattern property}: if $A$, $B$ and $C$ are completely disjoint finite subsets of $\Theta$ and $\alpha$ is an ordinal of $M$, such that every element of $A$ has value less than $\alpha$ and every element of $B$ has value greater than $\alpha$, then there is $v\in\Theta$ with value $\alpha$ such that $a\ttoward v$ and $v\ttoward b$ for all $a\in A$ and $b\in B$, and there are no $\ttoward$ relations between $v$ and elements of $C$, and furthermore, such that no nodes of $v$ arise on any sequence from $A$, $B$ or $C$.
\end{sublemma}

\begin{proof}
Let $n$ be least such that $\alpha\leq\lambda_n$. I shall first choose the terminal node $v_n$ of the desired surrogate sequence $v=\<v_0,\ldots,v_n>$. Let $B_n$ consist of the nodes arising in the sequences of $B$ with value in the interval $(\alpha,\lambda_n]$. These nodes arise either as the terminal nodes of sequences in $B$ that terminate with a node of value at most $\lambda_n$, or else they arise as the surrogate nodes at level $\lambda_n$ of a longer sequence. But in any case, $B_n\of\Gamma_n$, and furthermore, for $v$ to exhibit the correct $\ttoward$ relation to the elements of $B$, it will suffice that $v_n\toward b$ for each $b\in B_n$; and the main point of surrogates is that we shall be able to find such a $v_n$ inside $\Gamma_n$, without care for the much larger value nodes that may appear later on in the sequences of $B$. But we must also ensure the correct $\ttoward$ relation to $A$ and $C$, so let $A_n$ be the terminal nodes of any sequence in $A$ having value at least $\sup_{k<n}\lambda_k$, and let $C_n$ be the nodes occuring on a sequence in $C$ that are in $\Gamma_n$. Since $A$, $B$ and $C$ are completely disjoint, it follows that $A_n$, $B_n$ and $C_n$ are disjoint subsets of $\Gamma_n$, and every node in $A_n$ has value below $\alpha$ and every node in $B_n$ has value above $\alpha$. Thus, by the finite-pattern property of $\Gamma_n$, there is a new node $v_n\in\Gamma_n$ such that $a\toward v_n$ and $v_n\toward b$ for every $a\in A_n$ and $b\in B_n$ and such that $v_n$ has no $\toward$ relation with any node of $C_n$. Next, to define $v_k$ for $k<n$, let $A_k$ be the nodes arising from the sequences in $A$ in the $k^\th$ layer, that is, with value in the interval $[\sup_{i<k}\lambda_i,\lambda_k)$, and similarly let $C_k$ be any node on any sequence in $C$ in $\Gamma_k$. (We need not consider $B$ for defining the surrogates $v_k$ for $k<n$, since these surrogates act only as parents and not as children with respect to $\ttoward$.) By the finite-pattern property of $\Gamma_k$, there is a new node $v_k$ such that $a\toward v_k$ for each $a\in A_k$ and $v_k$ has no $\toward$ relation with any node in $C_k$. This defines $v=\<v_0,\ldots,v_n>$, which has value $\alpha$ and which consists of entirely new nodes not arising in $A$, $B$ and $C$. Furthermore, by design, $a\ttoward v$ for each $a\in A$ and $v\ttoward b$ for each $b\in B$, whilst $v$ has no $\ttoward$ relation with any sequence in $C$, and so we have fulfilled this instance of the surrogate finite-pattern property.
\end{proof}

So what we have is a set-like $\Ord^M$-graded digraph $\Theta\of L^M$, which does not have the finite-pattern property (and cannot by observation \ref{Observation.NoHomogeneousDigraphWithoutGrading}), but which has the surrogate finite-pattern property, and this is sufficient to carry out the universality construction. Specifically, we assign to each $x\in M$ a surrogate sequence $v_x$, with the same value as $\rank(x)$, ensuring that $x\in y\Iff v_x\ttoward v_y$. This can be achieved by enumerating $M=\set{x_n\st n<\omega}$, and choosing each $v_{x_n}$ so as to relate via $\ttoward$ to the previous $v_{x_k}$ for $k<n$ in exactly the same way that $x_n$ relates to $x_k$ via $\in$, while also having the right value and inductively maintaining that the sequences $v_{x_n}$ have no individual nodes in common. The surrogate finite-pattern property of the lemma exactly ensures that this recursive construction may proceed, and thus we build an embedding of $\<M,{\in^M}>$ into an induced subgraph of $\<\Theta,{\ttoward}>$.

Finally, I argue that $\<\Theta,{\ttoward}>\iso\<A,{\in^M}>$ for some $A\of L^M$, using the fact that $\ttoward$ is set-like in $L^M$. Specifically, since each $\Theta_n$ is a set in $L^M$, where $\ttoward$ is well-founded, we may inside $L^M$ carry out the modified Mostowski collapse of lemma \ref{Lemma.RealizingWellfoundedDigraphsAsSets}, defining $\pi(v)=\set{\pi(w)\st w\ttoward v}\union\set{\set{\emptyset,v}}$ to find $\<\Theta_n,{\ttoward}>\iso\<A_n,{\in^M}>\in L^M$. Furthermore, since each $\Theta_n$ is a $\ttoward$ initial segment of $\Theta$, it follows that these various isomorphisms cohere, and so $\<\Theta,{\ttoward}>\iso\<A,{\in^M}>$, where $A=\Union_n A_n\of L^M$.

Combining the two previous paragraphs, observe that $M$ is isomorphic to an induced subgraph of $\Theta$, which is isomorphic to a submodel of $L^M$, and so $M$ is isomorphic to a submodel of $L^M$, as desired.
\end{proof}

\begin{corollary}
 Every countable model $\<M,{\in^M}>$ of set theory is universal for all countable $\Ord^M$-graded binary relations.
\end{corollary}

\begin{proof}
This is what the proof of theorem \ref{Theorem.MisSubmodelOfL^M} establishes. The surrogate finite-pattern property of $\Theta$ ensures that it is universal for all countable $\Ord^M$-graded binary relations.
\end{proof}

\begin{corollary}\label{Corollary.MintoNiffOrd^MintoOrd^N}
 A countable model $\<M,{\in^M}>$ of set theory embeds into another model $\<N,{\in^N}>$ of set theory if and only if the ordinals of $M$ map order-preservingly into the ordinals of $N$.
\end{corollary}

\begin{proof}
We needn't assume $N$ is countable here, since we may simply pass to a countable submodel. For the forward implication, if $j:M\to N$ is an $\in$-embedding, then $f(\alpha)=\rank(j(\alpha))$ is an order-preserving map from $\Ord^M$ to $\Ord^N$. Conversely, if the ordinals of $M$ map into the ordinals of $N$, then we may externally view $\<M,{\in}^M>$ as an $\Ord^N$-graded digraph, and the previous corollary shows that $\<N,{\in^N}>$ is universal for such relations.
\end{proof}

\begin{corollary}\label{Corollary.IllfoundedModelsAreUniversal}
 Every ill-founded countable model of set theory $\<M,{\in^M}>$ is universal for all countable acyclic binary relations.
\end{corollary}

\begin{proof}
If $M$ is not an $\omega$-model, then already the hereditary finite sets $\HF^M$ of $M$ are universal by theorem \ref{Theorem.NonstandardModelsAreUniversal}. If $M$ has a standard $\omega$, but the ordinals are ill-founded, then by results in \cite{Friedman1973:CountableModelsOfSetTheories}, it follows that $\Ord^M\iso \lambda\cdot(1+\Q)$ for some admissible ordinal $\lambda$. In particular, $\Ord^M$ contains a countable dense order, and so by lemma \ref{Lemma.CycleFreeImpliesQgraded} every acyclic binary relation can be $\Ord^M$-graded. Thus, $\<M,{\in^M}>$ is universal for all such relations.
\end{proof}

\begin{theorem}[Main Theorem 2]\label{Theorem.CountableModelsLinearlyOrdered}
 The countable models of set theory are linearly pre-ordered by embeddability: for any two countable models of set theory $\<M,{\in^M}>$ and $\<N,{\in^N}>$, either $M$ is isomorphic to a submodel of $N$ or conversely. Indeed, the countable models of set theory are pre-well-ordered by embeddability in order type exactly $\omega_1+1$.
\end{theorem}

\begin{proof}
Corollary \ref{Corollary.IllfoundedModelsAreUniversal} shows that the models having ill-founded ordinals are all universal, so that every countable model of set theory embeds into any of them. In particular, they are all bi-embeddable with each other. What remains are the well-founded models, which by corollary \ref{Corollary.MintoNiffOrd^MintoOrd^N} are ordered the same as their ordinals. In particular, any two well-founded models with the same ordinals are bi-embeddable, and any shorter model is embeddable into any taller model.

Thus, the bi-embeddability classes of the models of set theory are in one-to-one correspondence with the collection of ordinals $\lambda$ such that $L_\lambda$ is a model of set theory, plus one more bi-embeddability class at the top for the ill-founded models. Since there are $\omega_1$ many countable admissible ordinals, it follows that the order-type of the bi-embeddability classes of the countable models of set theory (understood as models of \KP) is precisely $\omega_1+1$.
\end{proof}

If one wants to consider only the countable models of \ZFC, or only the models of \ZFC\ plus large cardinals, then the argument shows that they are pre-well-ordered in order type $\eta+1$, where $\eta$ is the number of countable ordinals $\lambda$ for which there is a model of the theory having height $\lambda$. If there is an inaccessible cardinal, for example, then it follows that $\eta=\omega_1$ for \ZFC\ models and the countable models of \ZFC\ are pre-well-ordered by embeddability in order type $\omega_1+1$.

Finally, let me observe that one definitely doesn't have linearity for the embeddability relation in the context of uncountable models.

\begin{observation}
If \ZFC\ is consistent, then there are two models $M,N\satisfies\ZFC$, such that neither is isomorphic to a submodel of the other.
\end{observation}

\begin{proof}
If \ZFC\ is consistent, then on the one hand, there is an uncountable model $M\satisfies\ZFC$ with cofinality $\omega$ and with every cut
of cofinality $\omega$. On the other hand, one may also build an $\omega_1$-like model
$N\satisfies\ZFC$, so $N$ is uncountable, but every initial segment is countable. It follows that $N$ has cofinality $\omega_1$.

Note that $M$ cannot embed into $N$, since there are elements of $M$ with uncountably many predecessors, but every element of $N$ has only
countably many predecessors. And $N$ cannot embed into $M$, since the supremum of the image would be a cut of uncountable cofinality.
\end{proof}

\section{Questions}

Although the main theorem shows that every countable model of set theory embeds into its own constructible universe
 $$j:M\to L^M,$$
this embedding $j$ is constructed completely externally to $M$ and there is little reason to expect that $j$ could be a class in $M$ or otherwise amenable to $M$.  To what extent can we prove or refute the possibility that $j$ is a class in $M$? This amounts to considering the matter internally as a question about $V$. Surely it would seem strange to have a class embedding $j:V\to L$ when $V\neq L$, even if it is elementary only for quantifier-free assertions, since such an embedding is totally unlike the sorts of embeddings that one usually encounters in set theory. Nevertheless, I haven't been able to refute hypothesis, and the possibility that there could be such an embedding is intriguing.

\begin{question}\label{Question.j:VtoL?}
 Can there be an embedding $j:V\to L$ when $V\neq L$?
\end{question}

By embedding, I mean an isomorphism from $\<V,{\in}>$ to its range in $\<L,{\in}>$, a map $j:V\to L$ for which $x\in y\Iff j(x)\in j(y)$. (Note that if $V=L$, then we have some easily defined non-identity embeddings $j:L\to L$ arising as in theorem \ref{Theorem.j:VtoV}.) Question \ref{Question.j:VtoL?} is most naturally formalized in \Godel-Bernays set theory, asking whether there can be a \GB-class $j$ forming such an embedding. If one wants $j:V\to L$ to be a definable class, then this of course implies $V=\HOD$, since the definable $L$-order can be pulled back to $V$, via $x\leq y\Iff j(s)\leq_L j(y)$. More generally, if $j$ is merely a class in \Godel-Bernays set theory, then the existence of an embedding $j:V\to L$ implies global choice, since from the class $j$ we can pull back the $L$-order. For these reasons, we cannot expect every model of \ZFC\ or of \GB\ to have such embeddings. Can they be added generically? Do they have some large cardinal strength? Are they outright refutable?

It they are not outright refutable, then it would seem natural that these questions might involve large cardinals; perhaps $0^\sharp$ is relevant. But I am unsure which way the answers will go. The existence of large cardinals provides extra strength, but may at the same time make it harder to have the embedding, since it pushes $V$ further away from $L$. For example, it is conceivable that the existence of $0^\sharp$ will enable one to construct the embedding, using the Silver indiscernibles to find a universal submodel of $L$; but it is also conceivable that the non-existence of $0^\sharp$, because of covering and the corresponding essential closeness of $V$ to $L$, may make it easier for such a $j$ to exist. Or perhaps it is simply refutable in any case. The first-order analogue of the question is:

\begin{question}
 Can every set $A$ admit an embedding $j:\<A,{\in}>\to \<L,{\in}>$, even when $V\neq L$?
\end{question}

The main theorem shows that every countable set $A$ embeds into $L$. What about uncountable sets? Let us make the question extremely concrete:

\begin{question}
 Can $\<V_{\omega+1},{\in}>$ embed into $\<L,{\in}>$ when there are non-constructible reals? How about $\<P(\omega),{\in}>$ or $\<\HC,{\in}>$?
\end{question}

It is also natural to inquire about the nature of $j:M\to L^M$ even when it is not a class in $M$. For example, can one find such an embedding for which $j(\alpha)$ is an ordinal whenever $\alpha$ is an ordinal?  The embedding arising in the proof of theorem \ref{Theorem.MisSubmodelOfL^M} definitely does not have this feature, since every set in $M$ is mapped to a node in $\Theta$, which is mapped ultimately to $\pi(v)=\set{\pi(w)\st w\ttoward v}\union\singleton{\singleton{\emptyset,v}}$, and this latter set is never an ordinal.

\begin{question}
 Does every countable model $\<M,{\in^M}>$ of set theory admit an embedding $j:M\to L^M$ that takes ordinals to ordinals?
\end{question}

Probably one can arrange this simply by being a bit more careful with the modified Mostowski procedure in lemma \ref{Lemma.RealizingWellfoundedDigraphsAsSets}, as it is applied to $\Theta$ in the proof of the main theorem, theorem \ref{Theorem.MisSubmodelOfL^M}. And if this is correct, then numerous further questions immediately come to mind, concerning the extent to which we ensure more attractive features for the embeddings $j$ that arise in the main theorems. This will be particularly interesting in the case of well-founded models, as well as in the case of $j:V\to L$, as in question \ref{Question.j:VtoL?}, if that should be possible.

\begin{question}
 Can there be a nontrivial embedding $j:V\to L$ that takes ordinals to ordinals?
\end{question}

Finally, I inquire about the extent to which the main theorems of this article can be extended from the countable models of set theory to the $\omega_1$-like models:

\begin{question}\label{Question.Omega1Like}
 Does every $\omega_1$-like model of set theory $\<M,{\in^M}>$ admit an embedding $j:M\to L^M$ into its own constructible universe? Are the $\omega_1$-like models of set theory linearly pre-ordered by embeddability?
\end{question}

Some of these questions are now answered in part by an observation of Menachem Magidor, who proved that if $x^\sharp$ exists for every set $x$, then there is no nontrivial embedding $j:V\to L$. In further work, we've observed that in the forcing extension $V[G]$ obtained by adding $\kappa^+$ many Cohen subsets to $\kappa$, there is no embedding $j:V[G]\to L$, and indeed, no embedding $j:P(\kappa)\to L$. In particular, it is consistent with \ZFC\ that there is no embedding of $\<P(\omega),{\in}>$ into $\<L,{\in}>$. Answering the latter part of question \ref{Question.Omega1Like}, Victoria Gitman and I have now constructed $\omega_1$-like models of \ZFC\ that are incomparable by embeddability.

\bibliographystyle{alpha}
\bibliography{MathBiblio,HamkinsBiblio}
\end{document}